\documentclass[12pt]{article}

\usepackage[centertags]{amsmath}
\usepackage{amsfonts}
\usepackage{amssymb}
\usepackage{amsthm}
\usepackage{times}
\usepackage{enumerate}
\usepackage{enumitem}
\usepackage{latexsym}
\usepackage{newlfont}
\usepackage{color}
\usepackage{float}
\usepackage{diagbox}
\usepackage{hyperref}
\usepackage{longtable}
\usepackage{rotating}
\usepackage{multirow}
\usepackage{extarrows}
\usepackage[sort,compress,numbers]{natbib}
\usepackage{array}
\usepackage{booktabs,makecell,amsmath}

\newtheorem{theorem}{Theorem}[section]

\newtheorem{lemma}[theorem]{Lemma}
\newtheorem{proposition}[theorem]{Proposition}
\newtheorem{corollary}[theorem]{Corollary}

\numberwithin{equation}{section}

\newcommand{\sign}{\mathrm{sign}}

\usepackage{graphics}
\usepackage{tikz}
\usetikzlibrary{shapes,fit}
\tikzset{
    tn/.style={shape=circle, draw, inner sep=1pt, color=black!70},
    ns/.style={shape=circle, draw, inner sep=1pt, fill=black}
}

\allowdisplaybreaks

\title{On Ward Numbers and Increasing Schr\"oder Trees}

\author{Elena L. Wang$^{1}$ and Guoce Xin$^{2}$
\\[2mm]
{\small $^{1}$ Center for Applied Mathematics, Tianjin University, Tianjin, 300072, P.R. China}\\[-0.8ex]
{\small $^{2}$ School of Mathematical Sciences, Capital Normal University, Beijing, 100048, P.R. China}\\[-0.8ex]
{\small $^{1}$ Email address: ling\_wang2000@tju.edu.cn} \\[-0.8ex]
{\small $^{2}$ Email address: guoce\_xin@163.com} }

\begin{document}

\maketitle

\begin{abstract}
The Ward numbers $W(n,k)$ combinatorially enumerate set partitions with block sizes $\geq 2$ and phylogenetic trees (total partition trees). We prove that $W(n,k)$ also counts \emph{increasing Schr\"oder trees} by verifying they satisfy Ward's recurrence. We construct a direct type-preserving bijection between total partition trees and increasing Schröder trees, complementing known type-preserving bijections to set partitions (including Chen's decomposition for increasing Schr\"oder trees). Weighted generalizations extend these bijections to enriched increasing Schröder trees and Schröder trees, yielding new links to labeled rooted trees. Finally, we deduce a functional equation for weighted increasing Schr\"oder trees, whose solution
using Chen's decomposition leads to a combinatorial interpretation of a Lagrange inversion variant.
\end{abstract}

\noindent
\begin{small}
\end{small}

\noindent
\begin{small}
\end{small}

\section{Introduction}
The Ward numbers $W(n,k)$, introduced by Ward~\cite{Ward-1934} in his study of Stirling number representations as factorial sums (see OEIS sequences A134991, A181996, and A269939~\cite{OEIS}), satisfy the recurrence for $n \geq 1$, $k \geq 0$:
\begin{align}
\label{e-Ward}
W(n,k) =  k W(n-1,k) + (n+k-1) W(n-1,k-1)
\end{align}
with initial condition $W(0,k) = \delta_{0,k}$, where $\delta_{i,j}$ is the Kronecker delta. Direct verification shows $W(n,n) = (2n-1)!!$, and for $n \geq 1$ we have $W(n,0) = 0$, $W(n,1) = 1$, and $W(n,2) = 2^{n+1} - n - 3$. The row sums $W(n)$ yield sequence A003311~\cite{OEIS}, enumerating \emph{total partitions}. The term ``total partition'' originates from Schr\"oder's fourth problem~\cite{Schroder-1870} on parenthesis arrangements with associativity and commutativity constraints~\cite{Moon-1987,Riordan-1976}. As formalized by Stanley \cite{Stanley}, a total partition recursively decomposes a set into singletons through successive nontrivial partitions (each with $\geq 2$ blocks); for example, the set $[3]=\{1,2,3\}$ admits four total partitions.

Combinatorial interpretations of $W(n,k)$ include:
\begin{enumerate}[label=(\roman*)]
    \item Set partitions with block sizes $\geq 2$ \cite{Carlitz-1971};
    \item Phylogenetic trees (total partition trees) \cite{Schroder-1870,Steel,Price-Sokal-2020}.
\end{enumerate}
We provide a new interpretation (referred to as item (iii)) via \emph{increasing Schr\"oder trees}. 

Schr\"oder trees, introduced by Chen \cite{Chen}, are labeled rooted trees where the set of subtrees of any vertex is endowed with the structure of ordered partitions. Chen established a decomposition of these trees into meadows, providing a unified framework for tree enumeration and Lagrange inversion, as they generalize both rooted trees and plane trees. Motivated by a problem of Gessel, Sagan, and Yeh \cite{Gessel-Sagan-Yeh}, Chen \cite{Chen-1999} later defined \emph{increasing} Schr\"oder trees with a more intricate decomposition algorithm. Both algorithms preserve the \emph{type} of a Schr\"oder tree, defined as the partition type of its non-root vertices.

This paper is organized as follows. Section~\ref{Sec-2} introduces weighted increasing Schr\"oder trees, as they provide the combinatorial framework for our new interpretation of $W(n,k)$. We review Chen's decomposition algorithms and interpret Schr\"oder trees as enriched increasing Schr\"oder trees. Using Chen's bijections, we establish a type-preserving bijection between them and connect enriched increasing Schr\"oder trees to labeled rooted trees. Section~\ref{Sec-3} details three combinatorial interpretations of Ward numbers. We prove combinatorially that (iii) satisfies recurrence~\eqref{e-Ward}. While bijections between (i) and (ii) are given by Erd\H{o}s--Sz\'ekely~\cite{Erdos-Szekely-1989} and Haiman--Schmitt~\cite{Haiman-Schmitt-1989}, and Chen's algorithm links (i) and (iii), we construct a \emph{direct} type-preserving bijection between (ii) and (iii). Weighted Ward numbers are also considered. Finally, Section~\ref{Sec-4} links Chen's decomposition to a variant of Lagrange inversion: the tree structure induces a functional equation whose solution via Chen's decomposition algorithm yields the combinatorial interpretation of Ward numbers, thus bridging enumerative tree theory with analytic inversion identities.

\section{Weighted increasing Schr\"oder trees}
\label{Sec-2}
This section reviews Chen's two decomposition algorithms: one for Schr\"oder trees and another for increasing Schr\"oder trees. The weighted version of the latter structure plays a fundamental role. We establish bijections between enriched increasing Schr\"oder trees and Schr\"oder trees, and between signed enriched increasing Schr\"oder trees and labeled rooted trees.

\subsection{Chen's two decomposition algorithms}
We begin by recalling essential terminology. An increasing tree is a labeled rooted tree where labels increase along every path from the root.
A Schr\"oder tree is a labeled rooted tree in which the subtrees of each vertex are endowed with an ordered partition structure.
It is called an increasing Schr\"oder tree if it is also an increasing tree. The height of a rooted tree is the number of edges on the longest path from the root to a leaf. A small tree is a rooted tree of height one. A meadow is a forest of small trees; it is increasing if all its small trees are increasing.

In what follows, the \emph{weight} of an object is always defined as the product of individual weights, and the weight of a set of objects is defined
as the sum of the weights of its elements, unless specified otherwise. For a Schr\"oder tree $T$ (increasing or not), we assign a weight $g_i$ to each block of size $i$ for all $i$. If the weight of $T$ is $w(T)=g_1^{m_1}g_2^{m_2}\cdots$, then we say $T$ has type $1^{m_1}2^{m_2}\cdots$, meaning $T$ contains $m_i$ blocks of size $i$ for each $i$. For a meadow, we assign the weight $g_i$ to each small tree on $i+1$ vertices, and its type is defined analogously.

Chen \cite{Chen} first introduced Schr\"oder trees to provide a combinatorial interpretation via a sign-reversing involution for cancellations occurring in the Lagrange inversion formula. Chen's first decomposition algorithm is given by the following bijection.
\begin{theorem}[Chen]
There exists a type-preserving bijection $\overline \phi$ from Schr\"oder trees with $n$ vertices and $k$ blocks to meadows with $n + k - 1$ vertices and $k$ small trees.\label{thm:2.1}
\end{theorem}

For example, Figure~\ref{tree_decomposition} shows a Schr\"oder tree of type $(2,1)$ (left) and its image under $\overline \phi$, a meadow of the same type $(2,1)$ (right). This bijection is remarkably general, specializing to known bijections for specific assignments of the weights $g_i$.

\begin{figure}[htbp]
\centering
\begin{minipage}{0.45\textwidth}
\centering
\begin{tikzpicture}[
  level 1/.style = {level distance=13mm, sibling distance=15mm},
  level 2/.style = {level distance=13mm, sibling distance=15mm}
]
\node [ns,label=90:$1$]{}[grow=down]
    child {node [ns,label=180:{$2$}](2){}
        child {node [ns,label=180:{$3$}](3){}}
        child {node [ns,label=0:{$4$}](4){}}
    };

\draw (2) ellipse (1.6em and 0.7em);
\node [ellipse, draw, fit=(3) (4), y radius=3em, x radius=1em] {};
\end{tikzpicture}
\end{minipage}
\hfill
\begin{minipage}{0.08\textwidth}
  \centering
  \tikz{\node[scale=1.5]{$\longrightarrow$};}
\end{minipage}%
\hfill
\begin{minipage}{0.45\textwidth}
\centering
\begin{tikzpicture}[
  level 1/.style = {level distance=13mm, sibling distance=15mm},
  level 2/.style = {level distance=13mm, sibling distance=15mm}
]
\node [ns,label=90:$2$]{}[grow=down]
    child {node [ns,label=180:{$3$}](3){}}
    child {node [ns,label=0:{$4$}](4){}};
\hspace{6em}
\node [ns,label=180:$1$]{}[grow=down]
    child {node [ns,label=180:{$5$}](5){}};
\end{tikzpicture}
\end{minipage}
\caption{A Schr\"oder tree and its meadow decomposition.}
\label{tree_decomposition}
\end{figure}
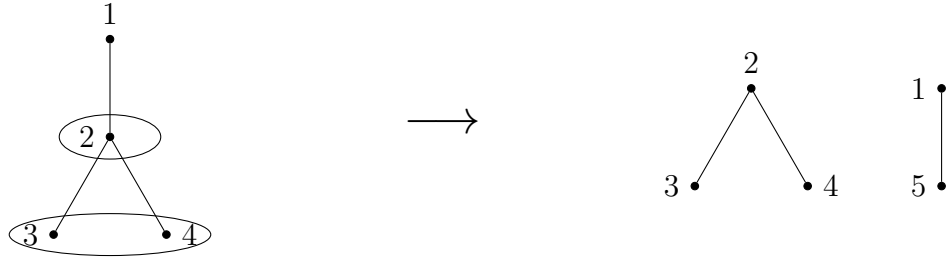

Motivated by questions from Gessel, Sagan, and Yeh \cite{Gessel-Sagan-Yeh} concerning tree enumeration by net inversion number, Chen \cite{Chen-1999} developed his second algorithm involving increasing structures. He introduced increasing Schr\"oder trees and established this decomposition.

\begin{theorem}[Chen]
There is a type-preserving bijection ${\phi}$ from the set of increasing Schr\"oder trees with $n$ vertices and $k$ blocks to the set of increasing meadows with $n + k - 1$ vertices and $k$ small trees.\label{thm:2.2}
\end{theorem}
Importantly, ${\phi}$ is not merely a restriction of $\overline \phi$ to increasing trees. Its construction is significantly more intricate and technical.

A direct consequence of the two bijections is the following result.
\begin{proposition}
\label{Prop-2.1}
The total weight under $\{g_i\}_{i\geq 1}$ of Schr\"oder trees with $n$ vertices and $k$ blocks equals the total weight
of increasing Schr\"oder trees under $\{(i+1)g_i\}_{i\geq 1}$ with $n$ vertices and $k$ blocks.
\end{proposition}
\begin{proof}
The proposition follows from the meadow perspective using Chen's two bijections. Consider a set partition $P$ of $[n+k-1]$ without singleton blocks.
It suffices to show that the total weight for meadows over $P$ equals that of the unique increasing meadow over $P$.
This holds because for each block of size $i+1$ ($i \ge 1$),
there are $i+1$ ways to form a small tree, each contributing weight $g_i$,
yielding a total weight of $(i+1)g_i$ per block;
This matches the weight $(i+1)g_i$ of the unique increasing small tree for that block.
\end{proof}

We now consider specific choices of the weight sequence $g_i$.
The total weights of increasing Schr\"oder trees for different choices
of $g$ exhibit many interesting properties, as illustrated in Section \ref{Sec-3.2} on weighted Ward numbers.

In the next subsection, we study weighted increasing Schr\"oder trees for $g_i = i+1$ and its signed version $g_i = (-1)^{i+1}(i+1)$.

\subsection{Enriched increasing Schr\"oder trees}
An \emph{enriched increasing Schr\"oder tree} is an increasing Schr\"oder tree where each block is marked with a $*$ at one of $i+1$ possible positions for a block of size $i$: either before the first vertex or to the right of any vertex in the block.

\begin{theorem}
\label{IS-S}
    There is a type-preserving bijection between the set of enriched increasing Schr\"oder trees with $n$ vertices and $k$ blocks and Schr\"oder trees with $n$ vertices and $k$ blocks.
\end{theorem}
\proof
Given a Schr\"oder tree $\overline{T}$, apply Chen's first bijection to obtain a meadow $\overline{M} = \overline{\phi}(\overline{T})$. For each small tree $\overline{S}$
in $\overline{M}$, convert it to an increasing small tree and mark it with a $*$ as follows:
(i) if $\overline{S}$ is increasing, place a $*$ before the first leaf;
(ii) if the $j$th leaf is the smallest, swap it with the root and place a $*$ to the right of the original root (now a leaf).

Now we obtain an increasing meadow $M$, where each small tree is endowed with a $*$ structure on its leaves. Apply Chen's second bijection ${\phi}^{-1}$ to $M$ and carry the $*$ structure (by \cite[Theorem~3.5]{Chen-1999}, $\phi$ carries combinatorial structures). The result is
an increasing Schr\"oder tree with each block marked by a $*$, giving the enriched tree.

This is a type-preserving bijection as each step is invertible and preserves the block structure.
\qed

A signed Schr\"oder tree corresponds to $g_i = (-1)^{i+1}$. Thus a block of size $i$
has weight $(-1)^{i+1}$, and a Schr\"oder tree $T$ on $n$ vertices with blocks $B_1,\dots, B_k$
has sign $\prod_{j=1}^k (-1)^{|B_j|+1} = (-1)^{n+k-1}$. The sign of an (enriched) increasing Schr\"oder tree is defined analogously.

\begin{theorem}
\label{SSch}
There is a sign-reversing involution $\psi_n$ on signed Schr\"oder trees on $n$ vertices with the following properties:
\begin{enumerate}
  \item $\psi_n$ preserves the underlying rooted tree structure, hence applies to signed increasing Schr\"oder trees.
  \item The fixed points of $\psi_n$ are Schr\"oder trees where all blocks are singletons, and the children of each internal vertex have increasing labels. 
\end{enumerate}
Consequently, the signed count of Schr\"oder trees on $n$ vertices is $n^{n-1}$, and the signed count of increasing Schr\"oder trees on $n$ vertices
is $(n-1)!$.
\end{theorem}

\begin{proof}
We construct $\psi_n$ recursively. The base case $\psi_1$ is trivial. Assume
$\psi_m$ is defined for all $m < n$.
Given a Schr\"oder tree $T$ on $n$ vertices,
let the root have children $v_1, \dots, v_s$ with subtrees $T_i$ on $m_i$ vertices rooted at $v_i$. Then $m_i < n$ for each $i$.
Apply $\psi_{m_1}$ to $T_1$: if $T_1$ is not a fixed point,
define $\psi_n(T)$ by replacing $T_1$ with $\psi_{m_1}(T_1)$;
otherwise, if $T_1,\dots,T_{i-1}$ are fixed points but $T_i$ is not,
define $\psi_n(T)$ by replacing $T_i$ with $\psi_{m_i}(T_i)$.

For the case where all $T_1, \dots, T_s$ are fixed points,
define an auxiliary map $\psi'$ acting on the ordered partition of the set of children $\{v_1, \dots, v_s\}$ as in Lemma \ref{Grigory} below.
Then $\psi_n(T)$ preserves each subtree but applies $\psi'$ to the partition of $\{v_1, \dots, v_s\}$.

The properties of $\psi_n$ follow from the construction. The consequences hold since:
fixed-point Schr\"oder trees correspond to rooted trees (counted by $n^{n-1}$);
fixed-point increasing Schr\"oder trees correspond to increasing rooted trees (counted by $(n-1)!$).
\end{proof}

Combining Theorem \ref{SSch} and Proposition \ref{Prop-2.1} yields:
\begin{corollary}
The signed count of enriched increasing Schr\"oder trees on $n$ vertices is $n^{n-1}$.
\end{corollary}

Let $[n] = \{1, 2, \dots, n\}$ and $\mathrm{OP}(n)$ be the set of ordered partitions of $[n]$, i.e., sequences $[B_1, \dots, B_k]$ where $B_i$ are nonempty, disjoint, and $\bigcup_i B_i = [n]$. The sign of $B_i$ is $\sign(B_i) = (-1)^{|B_i| + 1}$, and the sign of $op = [B_1, \dots, B_k]$ is
\[
\sign(op) = \prod_{i=1}^k \sign(B_i) = (-1)^{n + k}.
\]

\begin{lemma}
\label{Grigory}
For all integers $n \geq 1$, there is a sign-reversing involution $\psi'$ on $\mathrm{OP}(n)$ with the unique fixed point $[\{1\}, \{2\}, \dots, \{n\}]$.
\end{lemma}
This lemma is due to Grigory, as we will explain. Here we present a direct construction of the involution.

\begin{proof}
Let $\widetilde{OP}(n)$ denote the set $OP(n) \setminus \{[\{1\}, \{2\}, \dots, \{n\}]\}$, that is, all ordered partitions of $[n]$ except the one consisting of $n$ singleton blocks in natural order.
It suffices to define a sign-reversing involution $\psi': \widetilde{OP}(n) \to \widetilde{OP}(n)$ that pairs each ordered partition with another of opposite sign.

Given $op \in \widetilde{OP}(n)$, let $i$ be the largest index such that $B_\ell = \{\ell\}$ for $\ell = 1,2,\dots, i$. Note that $i \leq n-2$ because $op \neq [\{1\}, \{2\}, \dots, \{n\}]$. Then the element $i+1$ lies in some block $B_s$ with $s \geq i+1$. We distinguish two cases based on whether $B_s = \{i+1\}$.

\begin{enumerate}[label=\textbf{Case \arabic*.}, leftmargin=2em]
    \item If $B_s = \{i+1\}$, define $\psi'(op)$ by merging blocks $B_{s-1}$ and $B_s$ into a single block. Note that $s > i+1$ by the maximality of $i$.

    \item If $B_s \neq \{i+1\}$, define $\psi'(op)$ by splitting $B_s$ into two blocks $B_s \setminus \{i+1\}$ and $\{i+1\}$.
\end{enumerate}

Clearly, $\psi'$ maps Case 1 elements to Case 2 elements, and vice versa. Moreover, $\psi'^2$ is the identity map and $\sign(\psi'(op)) = -\sign(op)$.
Thus $\psi'$ is the desired sign-reversing involution.
\end{proof}

For example, given $n=3$, there are $12$ elements in $\widetilde{OP}(3)$.

\begin{table}[htbp]
\centering
\setlength{\tabcolsep}{10pt}   
\renewcommand{\arraystretch}{1.25} 
\small
\caption{An example of the involution $\psi'$.}
\label{tab:example}
\begin{tabular}{@{}c c c c@{}}
\toprule
\boldmath$op$ & \boldmath$\psi'(op)$ & \boldmath$\sign(op)$ & \boldmath$\sign(\psi'(op))$ \\
\midrule
$[\{1\},\{3\},\{2\}]$ & $[\{1\},\{2,3\}]$ & $+$ & $-$ \\
$[\{2\},\{1\},\{3\}]$ & $[\{1,2\},\{3\}]$ & $+$ & $-$ \\
$[\{2\},\{3\},\{1\}]$ & $[\{2\},\{1,3\}]$ & $+$ & $-$ \\
$[\{3\},\{1\},\{2\}]$ & $[\{1,3\},\{2\}]$ & $+$ & $-$ \\
$[\{3\},\{2\},\{1\}]$ & $[\{3\},\{1,2\}]$ & $+$ & $-$ \\
$[\{1,2,3\}]$         & $[\{2,3\},\{1\}]$ & $+$ & $-$ \\
\bottomrule
\end{tabular}
\end{table}

The involution in Lemma~\ref{Grigory} yields a combinatorial proof of the identity
\begin{align}
\sum_{k=0}^n (-1)^k k! S(n,k) = (-1)^n,
\end{align}
where $S(n,k)$ denotes the Stirling numbers of the second kind.

A related question on how to give a combinatorial proof of this identity was raised on Math StackExchange.\footnote{See \texttt{https://math.stackexchange.com/questions/395139}} In that discussion, M. Grigory gave a bijective proof of the identity using the interpretation of \( k! S(n,k) \) as the number of ordered set partitions of an \( n \)-element set into \( k \) blocks. Indeed, he recursively defined an involution, which is equivalent to our non-recursive version.

\section{Ward numbers}
\label{Sec-3}
In this section, we discuss three combinatorial interpretations of the Ward numbers:
the first, due to Carlitz, counts set partitions where each block has at least two elements;
the second, due to Steel, is expressed in terms of total partitions;
and the third involves increasing Schr\"oder trees.
To the best of our knowledge, the third interpretation is new.

\subsection{Combinatorial interpretations of Ward numbers}
Ward numbers, recursively defined by \eqref{e-Ward}, admit various combinatorial interpretations. Here we introduce three of them and consider their weighted versions.

Let $S_2(n,k)$ denote the number of partitions of an $n$-element set into $k$ nonempty subsets, each of size at least two. We assign the weight $g_i$ to a block of size $i+1$.

A total partition of the set $[n]$ is a process that recursively partitions non-singleton blocks into at least two nonempty subsets until only singletons remain. This process has a natural representation by a semi-labeled rooted tree, which is a rooted tree with labeled leaves and unlabeled internal vertices. A semi-labeled rooted tree with each internal vertex having degree at least $2$ is referred to as a \emph{total partition tree}. For a direct correspondence between total partitions and total partition trees, see \cite{Stanley}. To recover the total partition from its tree representation, we associate each internal vertex with the set of its leaf descendants. We assign the weight $g_i$ to an internal vertex of degree $i+1$. Similar to the case for Schr\"oder trees, for the objects considered here, if the weight is $g_1^{m_1}g_2^{m_2}\cdots$, then the type is defined to be $1^{m_1}2^{m_2}\cdots$.

\begin{proposition}
The following quantities are all equal to the Ward number $W(n,k)$.
\begin{enumerate}
  \item $S_2(n+k,k)$;
  \item The number of total partition trees on $[n+1]$ with $k$ internal vertices;
  \item The number of increasing Schr\"oder trees with $n+1$ vertices and $k$ blocks.
\end{enumerate}
\end{proposition}

The first item is due to Carlitz \cite{Carlitz-1971}, who established $W(n,k) = S_2(n+k,k)$ by comparing expressions for Stirling numbers of the second kind involving sums over either $W(n,k)$ or $S_2(n,k)$. The second item is well-known; see \cite{Price-Sokal-2020}.

Here we provide a combinatorial argument for items 1 and 3 from the recurrence perspective.
\begin{proof}[Recurrence proof of items 1 and 3]
For item 1, it suffices to show that for $n \ge 2$, $k \ge 1$,
\begin{align*}
    S_2(n,k)= k\,S_2(n-1,k)+(n-1)\,S_2(n-2,k-1),
\end{align*}
with initial conditions $S_2(n,0)=\delta_{n,0}$ for $n \ge 0$, $S_2(0,k)=\delta_{k,0}$, and $S_2(1,k)=0$ for $k \ge 0$.

Consider a partition of $[n]$ into $k$ subsets. We examine the position of element $n$.

\textbf{Case 1:} $n$ lies in a block of size at least three. Removing $n$ yields a partition of $[n-1]$ into $k$ subsets. Since $n$ can be reinserted into any of the $k$ blocks, this case contributes $kS_2(n-1,k)$ partitions.

\textbf{Case 2:} $n$ lies in a block of size two. Removing the block containing $n$ leaves a partition of an $(n-2)$-element set into $k-1$ subsets. The $n-1$ choices for the element paired with $n$ yield $(n-1)S_2(n-2,k-1)$ partitions.

The cases are disjoint and exhaustive. Verification of initial conditions is straightforward. This completes the proof for item 1.

For item 3, denote by $T(n,k)$ the number of increasing Schr\"oder trees with $n+1$ vertices and $k$ blocks. We show that $W(n,k)$ and $T(n,k)$ satisfy the same recurrence relations and initial conditions. Define $T(0,0)=1$ and $T(n,0)=0$ for $n \ge 1$. For $n \ge 1$ and $k \ge 1$, we establish
\begin{align*}
    T(n,k) = k T(n-1,k) + (n+k-1) T(n-1,k-1).
\end{align*}

Consider the block containing the leaf $n+1$ in an increasing Schr\"oder tree with $k$ blocks. The case where this block is not a singleton is counted in $kT(n-1,k)$ ways, with the factor $k$ corresponding to the choices for the block into which $n+1$ is inserted. The case where the block is a singleton is counted in $(n + k - 1)T(n-1,k-1)$ ways. The factor $(n + k - 1)$ arises from two alternatives: either the block containing $n+1$ is the leftmost child of its parent node (giving $n$ choices for the parent node), or it is immediately to the right of another block of the parent node (giving $k - 1$ choices for the adjacent block).
\end{proof}

Additionally, Price and Sokal \cite{Price-Sokal-2020} interpreted Ward numbers using augmented perfect matchings. Their work develops the recurrence relation for augmented perfect matchings and establishes a bijection between augmented perfect matchings and phylogenetic trees, which is a well-known interpretation for Ward numbers.

Chen's second decomposition $\phi$ explains the equality between items 1 and 3. Moreover, $\phi$ is type-preserving.

We observe that the equality between items 1 and 2 follows from the following bijection, which was independently found by Erd\H{o}s and Sz\'ekely \cite{Erdos-Szekely-1989} and by Haiman and Schmitt \cite{Haiman-Schmitt-1989}.
\begin{theorem}[Erdos-Sz\'ekely, Haiman-Schmitt]
There is a type-preserving bijection between the set of semi-labeled rooted trees with $ k $ unlabeled internal vertices and $ n+1 $ labeled leaves and the set of partitions of $ k $ blocks on $ n + k  $ elements. \label{thm:semilabeled}
\end{theorem}

Note that an internal vertex of degree $1$ corresponds to a singleton block. Next we present a type-preserving bijection for items 2 and 3.

\begin{theorem}
For $n \ge 1$, there is a type-preserving bijection between the set of total partitions of $[n]$ whose total partition tree has $ k $ internal vertices (including the root) and the set of increasing Schr\"oder trees on $[n]$ with $k$ blocks.\label{thm:3.3}
\end{theorem}

\proof
We establish the bijection through the following explicit construction:

\noindent \textbf{($\Rightarrow$) From a total partition to an increasing Schr\"oder tree:}

Let $P$ be a total partition of $[n]$. We construct an increasing Schr\"oder tree $T$ recursively as follows. For the base case of a trivial total partition on a single element, the corresponding increasing tree is uniquely defined. For non-trivial partitions, let $\pi = \{B_1, B_2, \ldots, B_k\}$ be the first partition of $P$, with blocks ordered increasingly by their minimal elements $m_1, m_2, \ldots, m_k$.
\begin{itemize}
    \item Let $T_1, T_2, \ldots, T_k$ denote the increasing Schr\"oder trees corresponding to the total partitions induced by $P$ on each block $B_1, B_2, \ldots, B_k$.
    \item Construct an increasing Schr\"oder tree $T'$ with root $m_1$ and left-most block $\{m_2, \ldots, m_k\}$, where each $m_i$ represents the tree $T_i$ for $i = 2, \ldots, k$.
    \item Merge $T_1$ with $T'$ by identifying their roots and attaching $T_1$ as the rightmost subtree of $T'$.
\end{itemize}

\noindent \textbf{($\Leftarrow$) From an increasing Schr\"oder tree to a total partition:}

Let $T$ be an increasing Schr\"oder tree with $n$ vertices. The correspondence is straightforward for $n = 1$. For $n > 1$, we proceed recursively:
\begin{itemize}
    \item Let the left-most block of the root of $T$ be $\{m_2, m_3, \ldots, m_k\}$, where each $m_i$ is the root of a subtree $T_i$. Let $T_1$ be the subtree obtained by removing the left-most block of the root (along with its descendants). Note that each $T_i$ is an increasing Schr\"oder tree.
    \item The first partition is recovered as $\{T_1, T_2, \ldots, T_k\}$, with the internal order structure within each $T_i$ being disregarded.
    \item Repeat this procedure for each subtree $T_i$ containing more than one vertex.
\end{itemize}

This bijection demonstrates that each internal vertex in a total partition tree corresponds to a block in the increasing Schr\"oder tree. Consequently, the number of total partition trees with $n + k$ vertices and $k$ internal vertices equals the number $T(n-1,k)$ of increasing Schr\"oder trees with $n$ vertices and $k$ blocks. This confirms that the bijection is type-preserving. The proof is complete.
\qed

We illustrate this bijection with the following example. Figure \ref{ex-total-Schroder} provides a total partition of $[6]$ whose total partition tree has $4$ internal vertices (including the root) and the corresponding increasing Schr\"oder trees on $[6]$ with $4$ blocks, both of type $(2,1^3)$.
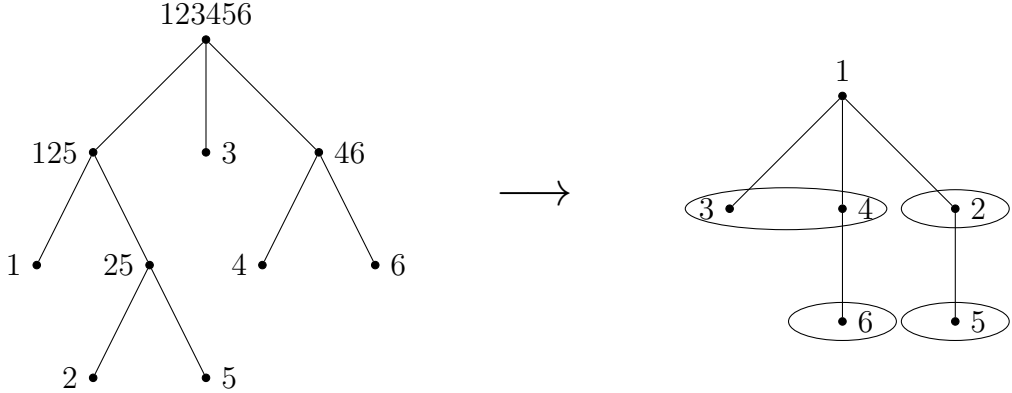
\begin{figure}[htbp]
\centering
\begin{minipage}{0.45\textwidth}
  \centering
  \begin{tikzpicture}[
      level 1/.style={level distance=15mm,sibling distance=15mm},
      level 2/.style={level distance=15mm,sibling distance=15mm},
      level 3/.style={level distance=15mm,sibling distance=15mm}]
  \node[ns,label=90:$123456$] {} [grow=down]
    child {node[ns,label=180:$125$] {}
      child {node[ns,label=180:$1$] {}}
        child {node[ns,label=180:$25$] {}
          child {node[ns,label=180:$2$] {}}
            child {node[ns,label=0:$5$] {}}
          }
          }
    child {node[ns,label=0:$3$]{}}
    child {node[ns,label=0:$46$] {}
    child {node[ns,label=180:$4$]{}}
          child {node[ns,label=0:$6$] {}}
          };
  \end{tikzpicture}
\end{minipage}%
\hfill
\begin{minipage}{0.08\textwidth}
  \centering
  \tikz{\node[scale=1.5]{$\longrightarrow$};}
\end{minipage}%
\hfill
\begin{minipage}{0.45\textwidth}
  \centering
  \begin{tikzpicture}[
      level 1/.style={level distance=15mm,sibling distance=15mm},
      level 2/.style={level distance=15mm,sibling distance=15mm}]
  \node[ns,label=90:$1$] {} [grow=down]
    child {node[ns,label=180:$3$](3) {}
          }
    child {node[ns,label=0:$4$](4) {}
    child {node[ns,label=0:$6$](6) {}}}
    child {node[ns,label=0:$2$](2) {}
    child {node[ns,label=0:$5$](5) {}}
          };
\node [ellipse, draw, fit=(3) (4), y radius=3em, x radius=1em] {};
\draw (2) ellipse (1.7em and 0.6em);
\draw (5) ellipse (1.7em and 0.6em);
\draw (6) ellipse (1.7em and 0.6em);
\end{tikzpicture}
\end{minipage}
\caption{Illustration of the bijection between total partitions and increasing Schr\"oder trees.}
\label{ex-total-Schroder}
\end{figure}

\subsection{Weighted version}
\label{Sec-3.2}
Let us define the weighted Ward number $W^g(n,k)$ for non-negative integers $n$ and $k$ as the total weight of partitions of the set $[n+k]$ into $k$ blocks, each of size at least two. The initial condition is given by $W^g(n,0) = \delta_{n,0}$ for all $n \geq 0$. Recall that the weight of a block of size $i+1$ is assigned  $g_i$. This concept can also be interpreted in terms of total partition trees and increasing Schr\"oder trees.

The following generating function expression can be derived directly:
\begin{equation}
\label{eq:W^g}
  W^g(n,k) = \left[\frac{ x^{n+k}}{(n+k)!}\right] \ \  \frac{1}{k!} \left( g_1\frac{ x^2}{2!} + g_2\frac{ x^3}{3!} + \cdots \right)^{\!\!k} ,
\end{equation}
where $[x^{n+k}]$ denotes the coefficient extraction operator.

To conclude this section, we summarize specializations of the sequence $\{g_i\}$ and their combinatorial interpretations. Here, $W^g(n)$ represents the sum of $W^g(n,k)$ over all $k$, while $\widetilde{W}^g(n)$ denotes the alternating sum $\sum_k (-1)^{n+k} W^g(n,k)$. The key results are presented in Table~\ref{tab:specializations}, with detailed formulas and references provided below.

\begin{table}[htbp]
\centering
\renewcommand{\arraystretch}{1.3}      
\setlength{\tabcolsep}{8pt}            
\small
\caption{Specializations of $g_i$ and combinatorial interpretations}
\label{tab:specializations}
\begin{tabular}{@{}ccccc@{}}
\toprule
\textbf{\boldmath$g_i$} &
\makecell[c]{\textbf{Sequence name of} \boldmath$W(n,k)$} &
\textbf{OEIS} &
\boldmath$W^g(n)$ &
\boldmath$\widetilde{W}^g(n)$ \\
\midrule
$1$                           & Ward set numbers       & A269939 & A000311             & $n!$ \\
$i+1$                         & Enriched Ward numbers  & A368584 & A053492             & $(n+1)^n$ \\
$i!$                          & Ward cycle numbers     & A269940 & A032188             & $1$ \\
$(i+1)!$                      & Weighted Ward numbers  & A357367 & A032037             & $(n+1)!$ \\
$(i-1)!$                      & ---                    & A239098 & A000312,\,$n^n$     & A074059 \\
$g_i=\delta_{i,1}$            & ---                    & ---     & A001147,\,$(2n-1)!!$& $W^g(n)$ \\
$g_i=2\delta_{i,1}$           & ---                    & ---     & A001813,\,$(2n)!/n!$& $W^g(n)$ \\
$g_i=g_j\delta_{i,j}$         & Partition coefficients & ---     &
$\displaystyle
\frac{(i(j+1))!}{i!\,((j+1)!)^i}\,g_j^{\,i}\delta_{n,ij}$ &
$(-1)^{n+i}W^g(n)\delta_{n,ij}$ \\
\bottomrule
\end{tabular}
\end{table}

Detailed descriptions:
\begin{enumerate}
   \item For $g_i = 1$, the ordinary Ward numbers (or Ward set numbers) are given by $W(n,k) = \sum_{m=0}^{k} (-1)^{m+k} \binom{n+k}{n+m} S(n+m,m)$.
The sum $W(n)$ is the number of total partitions of $n+1$ (A000311).
       Theorem \ref{SSch} proves combinatorially:
    \[
    \widetilde{W}(n) = \sum_{k} (-1)^{n+k} W(n,k) = n!.
    \]

    \item When $g_i = i+1$, the enriched Ward numbers is $\overline{W}(n,k) = k! \binom{n+k}{k} S(n,k)$ (A368584). The sum $\overline{W}(n)$ is A053492, counting
        Schr\"oder trees with $n+1$ vertices. Theorem \ref{SSch} proves combinatorially:
    \[
    \widetilde{\overline{W}}(n) = \sum_{k} (-1)^{n+k} \overline{W}(n,k)= \sum_k  k! \binom{n+k}{k} S(n,k) = (n+1)^n,
    \]
    which enumerates labeled rooted trees (A000169).

    \item Setting $g_i = i!$ yields the Ward cycle numbers $W^g(n,k) = \sum_{m=0}^{k} (-1)^{m+k} \binom{n+k}{n+m} |s(n+m,m)|$, where $s(n,k)$ are the signed
        Stirling numbers of the first kind. This is A269940. The total sum $W^g(n)$ is A032188, counting plane increasing trees on $n+1$ vertices where
        each vertex of degree $k \geq 1$ admits $2^{k-1}$ colorings. Additionally, the alternating sum is 1.

    \item If $g_i = (i+1)!$, then $W^g(n,k)$ is A357367, which can be written as $\sum_{m=0}^{k} (-1)^{m+k} \binom{n+k}{n+m} L(n+m,m)$, with $L(n+m,m)$ the
        unsigned Lah numbers (A271703). The sum $W^g(n)$ is A032037 and equals to $n!$ times the $(n-1)$-th little Schr\"oder number (A001003). Moreover, the
        alternating sum is $(n+1)!$.

    \item For $g_i = (i-1)!$, the numbers $W^g(n,k)$ are A239098, representing constant terms of polynomials related to Ramanujan's $\psi$ polynomials. The
        sum $W^g(n) = n^n$ is A000312. The alternating sum $\widetilde{W}^g(n)$ is A074059, giving the dimension of the cohomology ring of the moduli space of
        genus $0$ curves with $n$ marked points, subject to associativity equations in physics.

    \item When $g_1 = 1$ and $g_k = 0$ for $k > 1$, $W^g(n) = (2n-1)!!$ (A001147). This counts labeled plane increasing trees and solves Schr\"oder's third
        problem.

    \item When $g_1 = 2$ and $g_k = 0$ for $k > 1$, $W^g(n) = (2n)!/n!$ (A001813), counting labeled plane trees on $n+1$ vertices.

    \item In general, for fixed $j \geq 1$ with $g_i = 0$ for $i \neq j$, $W^g(n) = 0$ unless $n = ij$ for some integer $i \geq 0$. Then
    \[
    W^g(ij) = \frac{(i(j+1))!}{i! \, ((j+1)!)^i} \, g_j^i.
    \]
    The coefficient counts partitions of $i(j+1)$ labeled items into $i$ unlabeled boxes of size $j+1$ (A060540).
\end{enumerate}
We remark that in the last three cases, $W^g(n,k)$ is nonzero only for a particular $k$, therefore, the nonzero term of $W^g(n,k)$ is of the same value as $W^g(n)$.

\section{A variation of the Lagrange inversion formula}
\label{Sec-4}
There exists a less familiar reformulation of the Lagrange inversion formula, which is equivalent to the classical version. In this section, we establish a connection between this variation and the decomposition algorithm for increasing Schr\"oder trees.

Let \( f(x) = \sum_{n \geq 1} a_n \frac{x^n}{n!} \) be a formal power series with \( a_1 \neq 0 \), and let \( g(x) = \sum_{n \geq 1} b_n \frac{x^n}{n!} \) denote its compositional inverse, satisfying \( f(g(x)) = g(f(x)) = x \). The compositional inverse is also denoted by $ f^{<-1>}(x) $. The classical Lagrange inversion formula expresses:
$$ b_n =  \left[ \frac{x^{n-1}}{(n-1)!}\right] \quad \left( \frac{x}{f(x)} \right)^{\!\!n}. $$
We now discuss the following equivalent formulation of the Lagrange inversion formula. See Theorem 2.6.1 in the nice survey by Gessel \cite{gessel}
about the Lagrange inversion formula.

\begin{theorem}
Let $ h(x) $ be a formal power series with $ h(0) = 0 $, $ h'(0) = 1 $. Then we have
\begin{align}
    h^{<-1>}(x) = x + \sum_{k \geq 1} \frac{1}{k!} \left[ (x - h(x))^k \right]^{(k-1)}. \label{4.1}
\end{align}\label{thm:4.1}
\end{theorem}

\proof
Following the notations in Section 3.2, let \( g_1, g_2, \ldots \) be a sequence of indeterminates. The weight of a block with \( i \) vertices in a Schr\"oder tree is assigned \( g_i \). Denote by \( W_n \) the total weight of all increasing Schr\"oder trees on \( n \) vertices, and by \( V_n \) the weight for all increasing Schr\"oder trees on \( n \) vertices such that the root has only one block.

We aim to establish a functional equation for
\[
W(x) = \sum_{n \geq 1} \frac{1}{n!} W_n x^n.
\]
A combinatorial formula for \( W_n \) will lead to a solution of this functional equation.

Consider an increasing Schr\"oder tree \( T \) on \( n+1 \) vertices such that the root has \( k \) blocks. Let \( B_i \) be the set of vertices in the \( i \)-th block of the root of \( T \) along with all their descendants. Then \( B_1, B_2, \ldots, B_k \) form a partition of \( \{2, 3, \ldots, n+1\} \). By adding the element 1 as the root to the subtree restricted to \( B_i \), we obtain an increasing Schr\"oder tree where the root has only one block. This leads to the recurrence relation:
\begin{align}
W_{n+1} = \sum_{(B_1, B_2, \ldots, B_k)} V_{b_1+1} V_{b_2+1} \cdots V_{b_k+1}, \label{eq:4.2}
\end{align}
where \( (B_1, B_2, \ldots, B_k) \) ranges over all ordered partitions of \( \{2, \ldots, n+1\} \), and \( b_i = |B_i|\).

For an increasing Schr\"oder tree \( Q \) where the root has only one block with \( n+1 \) vertices, removing the root of \( Q \) gives the recurrence relation:
\begin{align}
V_{n+1} = \sum_{\{C_1, C_2, \ldots, C_k\}} g_k W_{c_1} W_{c_2} \cdots W_{c_k}, \label{eq:4.3}
\end{align}
where \( \{C_1, C_2, \ldots, C_k\} \) ranges over all unordered partitions of \( \{2, \ldots, n+1\} \), and \( c_j = |C_j| \).

Using generating function theory, equations \eqref{eq:4.2} and \eqref{eq:4.3} lead to:
\begin{align}
\sum_{n \geq 1} \frac{1}{n!} W_{n+1} x^n &= \sum_{k \geq 1} \left( \sum_{n \geq 1} \frac{1}{n!} V_{n+1} x^n \right)^{\!\!k}, \label{4.4} \\
\sum_{n \geq 1} \frac{1}{n!} V_{n+1} x^n &= g \left( \sum_{n \geq 1} \frac{1}{n!} W_n x^n \right), \label{4.5}
\end{align}
where
\[
g(x) = \sum_{n \geq 0} \frac{1}{n!} g_n x^n
\]
with the convention \( g_0 = 0 \). Define
\[
f(x) = \sum_{n \geq 1} \frac{1}{n!} g_{n-1} x^n.
\]
From equations \eqref{4.4} and \eqref{4.5}, we derive the differential equation:
\[
W'(x) = \frac{1}{1 - f'(W(x))},
\]
which can be rewritten as:
\[
W'(x) - f'(W(x)) W'(x) = 1.
\]
This leads to the equation:
\[
W(x) - f(W(x)) = x.
\]
Let \( h(x) = x - f(x) \). Then \( h(0) = 0 \), \( h'(0) = 1 \), and importantly, \( W(x) = (h(x))^{<-1>} \).

We now compute $W(x)$ in an alternative manner. By applying the type-preserving bijection established in Theorem \ref{thm:2.2}, we observe that $W_n$ can be expressed as the sum of $W_{n,k} = W^g(n-1,k)$, as defined in \eqref{eq:W^g}. 
For $k \geq 1$, we have the generating function relation:
\begin{align}
\sum_{n \geq 2} \frac{1}{(n + k - 1)!} W_{n,k} x^{n+k-1} = \frac{1}{k!} \left( \sum_{n \geq 1} \frac{1}{n!} g_{n-1} x^n \right)^{\!\!k}=\frac{1}{k!} f^k.
\label{4.7}
\end{align}

By differentiating both sides of \eqref{4.7} $(k-1)$ times, we obtain the relation:
$$
\sum_{n \geq 2} \frac{1}{n!} W_{n,k} x^n = \left( \frac{1}{k!} f^k \right)^{\!\!(k-1)}.
$$
Summing over all $k$, we arrive at the expression:
\begin{align}
W(x) = x + \sum_{k \geq 1} \left( \sum_{n \geq 2} \frac{1}{n!} W_{n,k} x^n \right) = x + \sum_{k \geq 1} \frac{1}{k!} (f^k)^{(k-1)}.
\end{align}
This completes the derivation of the desired formula for the inverse of $h(x)$.
\qed

It is worth noting that the restriction on the coefficient of $x$ in $h(x)$ from the previous theorem can be removed by utilizing the relation $[a h(x)]^{<-1>} = h^{<-1>}(x/a)$ for any nonzero constant $a$. In the context of Theorem \ref{thm:4.1}, setting $g_i = 1$ for all $i$ reduces $W_n$ to the number of increasing Schr\"oder trees with $n$ vertices. Consequently, the generating function for $W_n$ becomes $W(x) = (1 + 2x - e^x)^{<-1>}$, which coincides with the generating function for total partitions. As a consequence of Theorem \ref{thm:4.1}, we have the formula
\begin{align}
W_n=\sum_{1 \le i \leq k \leq n} (-1)^{k-i} S(n+i-1,i) \binom{n+k-1}{n+i-1}.
\end{align}
The classical Lagrange inversion formula also gives the following infinite sum formula:
\begin{align}
    W_n = \sum_{k \geq 0} \frac{1}{2^{n+k}} S(n+k-1,k),
\end{align}
where recall that $S(0,0) = 1$ and $S(n,0) = 0$ for $n \geq 1$.

\section{Concluding Remarks}
We have provided an interpretation of the Ward numbers $W(n,k)$ using increasing Schr\"oder trees and considered their weighted counterparts. The formulas for weighted Ward numbers presented in Section \ref{Sec-3.2} can be derived through generating functions, specifically as the compositional inverse of an explicit generating function within the framework of Theorem \ref{thm:4.1}. Observing the results in Table~\ref{tab:specializations}, it appears that $\widetilde{W}^g(n)$ admits a particularly elegant formula. We have established combinatorial proofs for both the ordinary and enriched cases. Exploring combinatorial proofs for the remaining cases might be an interesting direction for further research.

We propose the following open problem: Find a direct bijective proof of Theorem \ref{IS-S}. That is, construct a direct bijection between enriched increasing Schr\"oder trees and Schr\"oder trees. While these two structures share similarities, they exhibit significant differences. Our current proof heavily relies on Chen's second decomposition, which is elegant but technical.

\noindent \textbf{Acknowledgement.} The work was supported by the National Science Foundation of China.

\end{document}